\documentclass[12pt,leqno]{article}
\usepackage[dvips]{graphicx}
\usepackage{amsmath,amssymb,amsthm,amscd}
\usepackage[mathscr]{eucal}
\usepackage{amsfonts}
\begin{document}
\parskip=6pt
\numberwithin{equation}{section}
\newtheorem{theorem}{Theorem}[section]
\newtheorem{corollary}[theorem]{Corollary}
\newtheorem{lemma}[theorem]{Lemma}
\newtheorem{definition}[theorem]{Definition}





\newcommand{\cF}{\cal F}
\newcommand{\cA}{\cal A}
\newcommand{\cC}{\cal C}
\newcommand{\cO}{\cal O}
\newcommand{\cJ}{\cal J}
\newcommand{\var}{\varepsilon}
\newcommand{\bC}{\mathbb C}
\newcommand{\bP}{\mathbb P}
\newcommand{\bN}{\mathbb N}
\newcommand{\bA}{\mathbb A}
\newcommand{\bR}{\mathbb R}
\newcommand{\fU}{\frak U}
\newcommand{\hol}{\text{hol}}
\newcommand{\Hom}{\text{Hom}}
\newcommand{\id}{\text{id}}
\newcommand{\Ker}{\text{Ker }}
\renewcommand\qed{ }
\begin{titlepage}
\title{\bf A Maximum Principle for Hermitian (and other) Metrics\thanks{Research partially supported  by  NSF grant DMS-1162070 \newline
2010 Mathematics Subject Classification 32L10, 32Q99, 32U05}}
\author{L\'aszl\'o Lempert\\ Department of  Mathematics\\
Purdue University\\West Lafayette, IN
47907-2067, USA}
\end{titlepage}
\date{}
\maketitle
\abstract
We consider homomorphisms of hermitian holomorphic Hilbert bundles. Assuming the homomorphism decreases curvature, we prove that its pointwise norm is plurisubharmonic.

\endabstract

\section{Introduction}

Suppose $E\to S$ is a holomorphic Hilbert bundle endowed with hermitian metrics $h^1$ and $h^2$, such that the curvature of $h^2$ is less than the curvature of $h^1$, in a sense to be made precise below. In this paper we will show that this implies that $h_2:h_1$ satisfies the maximum principle: if $U\subset S$ is a  relatively compact subdomain, $\partial U\ne\emptyset$, and $h^2\le h^1$ over $\partial U$, then $h^2\le h^1$ over $U$. A variant of this result, when $h^1$ or $h^2$ is flat, was key in Berndtsson's solution of the openness conjecture of Demailly--Koll\'ar \cite{B2}, \cite{DK}. Berndtsson tells me that for finite rank bundles he was aware of the maximum principle for non-flat metrics as well, in fact, the proof he suggested for [BK, Lemma 8.11] works in general finite rank bundles.

The natural framework of this sort of problems is even more general than Hilbert bundles. We will set up this framework and formulate our main result, Theorem 2.4, in section 2. Here we state a special case that applies to homomorphisms of Hilbert bundles.
The meaning of the notions involved should be at least intuitively clear by analogy with finite rank bundles; at any rate the terms will be defined in section 2.

\begin{theorem}Let $S$ be a complex manifold, and $E, E'\to S$ holomorphic  Hilbert bundles endowed with hermitian metrics $h,h'$ of class $C^2$.  Let $R$ and $R'$ denote the curvature operators of the (Chern connections defined by the) metrics, so that e.g. $R:(\bC\otimes TS)^2\to {\text{End }}H$. Let $A:E\to E'$ be a holomorphic homomorphism, and let $\|A\|:S\to [0,\infty)$ denote its pointwise operator norm. If $A$ decreases curvature: for $\xi\in T_s^{1,0}S$ and $v\in E_s$ such $Av\ne 0$
\begin{equation}
\frac{h'\left(R'(\xi,\bar\xi)Av,Av\right)}{h'(Av, Av)}\le\frac{ h\left(R(\xi,\bar\xi)v,v\right)}{h(v,v)},
\end{equation} 
then $\log\|A\|$ is plurisubharmonic, hence satisfies the maximum principle. 
\end{theorem}

The maximum principle that Berndtsson proved corresponds to $E=E'$, $A=\text{id}$ (and $h$ flat).

As far as we can tell, Ahlfors was the first to notice the connection between curvature estimates for holomorphic maps and subharmonicity, which he exploited in \cite{A} to obtain a generalization of Schwarz's lemma.
The principle is the same here. The reader will notice that the conclusions in \cite{A} and in subsequent generalizations, e.g. by Royden and Yau \cite{R,Y}, are stronger than in our theorem in that they estimate $\|A\|$. This is because Ahlfors--Schwarz type results concern tangent bundles, the size of curvature is measured differently than in (1.1), and  as a result, the hypothesis of decreasing curvature is also somewhat different.

\section{Metrics on Bundles and on Sheaves of Infinite Rank}

A smooth, respectively holomorphic, Banach bundle is a smooth, respectively holomorphic, map $\pi:E\to S$ of smooth, respectively complex, Banach manifolds, with each fiber $E_s=\pi^{-1}\{s\}$ endowed with the structure of a complex vector space. It is required that for each $s\in S$ there be a neighborhood $U\subset S$, a complex Banach space $W$, and a diffeomorphism, respectively biholomorphism, $F:U\times W\to \pi^{-1}U$ (local trivialization) that maps $\{t\}\times W$ linearly on $E_t$, $t\in U$. Thus any holomorphic Banach bundle is automatically a smooth Banach bundle. We will only consider bundles over finite dimensional bases $S$ (although most of what follows will hold when $S$ is any Banach manifold, in particular, Theorem 2.4 will hold). When the Banach spaces $W$ above are Hilbert, we speak of a Hilbert bundle. For $W,Z$ Banach spaces we write $\text{Hom}(W,Z)$ for the Banach space of bounded linear operators $W\to Z$, endowed with the operator norm, and let End $W=\text{Hom}(W,W)$.

A connection $\nabla$ of class $C^k$, $k=0,1,\dots$, on a smooth Banach bundle $\pi: E\to S$ maps germs of $C^1$ sections of $E$ to germs of $E$-valued one forms. It is required that for trivializations $F:U\times W\to \pi^{-1}U$ as above there should exist an End $W$-valued one form $\theta$ on $U$, of class $C^k$, such that for $f\in C^1(U;W)$ corresponding to the section $\varphi(t)=F(t,f(t))$, and a vector $\xi\in\bC\otimes T_sU,\ s\in U$
$$
(\nabla\varphi)(\xi)=F(s, \xi f+\theta(\xi)f).$$
We also write $\nabla_{\xi}\varphi$ for $(\nabla\varphi)(\xi)$.

A connection $\nabla$ on a holomorphic Banach bundle is compatible with the holomorphic structure if the connection forms $\theta$ corresponding to holomorphic local trivializations are of type $(1,0)$, i.e., $\theta(\xi)=0$ when $\xi\in T^{0,1}U$.

A metric on a smooth Banach bundle is a locally uniformly continuous function $q:E\to[0,\infty)$ that restricts on each fiber $E_s$ to a norm that induces the topology of $E_s$. This amounts to requiring that with local trivializations $F:U\times W\to \pi^{-1}U$ the norms $\|\ \ \|_s=(q\circ F)(s,\cdot)$ on $W$ be all equivalent, and as $s$ varies, locally uniformly so.

Suppose now that $E\to S$ is a smooth Hilbert bundle. A hermitian metric on $E$, of class $C^k$, is a function $h:E\oplus E=\coprod_{s\in S} E_s\oplus E_s\to \bC$. We require for each local trivialization $F:U\times W\to \pi^{-1}U$ that there exist a $C^k$ map $P:U\to {\text{End}}\,  W$, taking values in invertible self adjoint operators, such that
\begin{equation}
h\left(F(s,w_1), F(s, w_2)\right)=\left\langle P(s)w_1, w_2\right\rangle,
\end{equation}
where $\langle\ , \rangle$ is the inner product on $W$. 

A metric $h$ of class $C^1$ on a holomorphic Hilbert bundle determines a Chern connection $\nabla$, the unique connection compatible both with the holomorphic structure and the metric, the latter in the sense that 
\begin{equation}
\xi h(\varphi,\psi)=h(\nabla_\xi\varphi,\psi)+h(\varphi, \nabla_{\bar\xi}\psi)
\end{equation}
for $C^1$ sections $\varphi,\psi$. Given a local trivialization and the corresponding $P:U\to {\text{End}}\, W$ as in (2.1), the connection form is $\theta=P^{-1}\partial P$.

Suppose now the metric $h$ is of class $C^2$, and $\xi\in T^{1,0}_s S$. The curvature operator
$$
R(\xi, \bar{\xi})=\nabla_\xi\nabla_{\bar\xi}-\nabla_{\bar\xi}\nabla_\xi,
$$
computed by first extending $\xi$ to a holomorphic vector field in a neighborhood of $s$, is independent of the extension, and defines a bounded linear operator on $E_s$ (see e.g. [LSz, Lemma 2.2.4]). If $\varphi,\psi$ are nonvanishing holomorphic sections near $s$, then from (2.2) one computes
\begin{multline*}
 \xi\bar\xi\log h(\varphi,\psi)= \\
-\frac{h(\varphi,R(\xi,\bar\xi)\psi)}{ h(\varphi,\psi)}
+\frac{h(\nabla_\xi\varphi,\nabla_\xi\psi)h(\varphi,\psi)-h(\nabla_\xi\varphi,\psi)h(\varphi,\nabla_\xi\psi)}{ h(\varphi,\psi)^2}.
\end{multline*}
The left hand side will not change if the order of $\xi$ and $\bar\xi$ is changed; on the right only the numerator of the first term will change to $h(R(\xi,\bar\xi){\varphi,\psi})$. This shows $R\left(\xi,\bar\xi\right)$ is self adjoint. As a special case we obtain
\begin{multline}
\xi\bar\xi\log h(\varphi,\varphi)=\\-\frac{h(\varphi,R(\xi,\bar\xi)\varphi)}{h(\varphi,\varphi)}
+\frac{h(\nabla_\xi\varphi,\nabla_\xi\varphi) h(\varphi,\varphi)-|h(\nabla_\xi\varphi,\varphi)|^2}{h(\varphi,\varphi)^2}.
\end{multline}

Curvature can be defined in much greater generality  than Hilbert bundles: for sheaves of $\cal O$-modules over complex manifolds, endowed with a not necessarily hermitian metric. Let $S$ be a (finite dimensional) complex manifold, ${\cF}\to S$ a sheaf of $\cal O$-modules, and ${\cal P}\to S$ the sheaf of germs of nonnegative continuous functions. Thus $\cal P$ is a partially ordered semiring; in the partial order $u\le v$ for $u, v\in{\cal P}_s$ means $v-u\in{\cal P}_s$.

\begin{definition}A metric on $\cal F$ is a morphism $p: \cal F\to\cal P$ of sheaves (considered as sheaves of sets) satisfying for $s\in S$ and $\varphi,\psi\in{\cal F}_s, f\in{\cal O}_s$
\begin{equation}
p(\varphi+\psi)\le p(\varphi)+p(\psi), \qquad p(f\varphi)=|f|p(\varphi),
\end{equation}
 and $\varphi=0$ if and only if $p(\varphi)=0$.
\end{definition}

The simplest example is $\cal F$ the sheaf of holomorphic sections of a holomorphic Banach bundle $E\to S$ that is endowed with a metric $q$, and $p(\varphi)=q(\varphi)$. More complicated examples arise from direct images of, say, hermitian holomorphic Hilbert bundles $(E, h)\to T$ under holomorphic submersions $\sigma:T\to S$. Given a continuous form $\nu$ on $T$ that restricts to volume forms on the fibers $\sigma^{-1}s, s\in S$, for each open $U\subset S$ and $\Phi\in{\cO} (\sigma^{-1} U, E)$, define 
${\widetilde p}(\Phi)(s)=\left(\int_{\sigma^{-1}s} h(\Phi,\Phi)\nu\right)^{1/2}$. Let 
\[
{\cF}(U)=\{\Phi\in{\cO}(\sigma^{-1}U, E):  {\widetilde p}(\Phi):U\to [0,\infty)\  \text{is continuous} \}.
\]
(It is not completely obvious that ${\cal F}(U)$ is closed under addition, but it is not hard to verify, either, see [LSz, Lemma 6.2.1].) By passing to germs, $\widetilde p$ induces a metric $p$ on the sheaf $\cF$ associated with the presheaf $ U\mapsto{\cal F}(U)$.

Suppose $\cal F$, ${\cal F}'\to S$ are sheaves of $\cal O$-modules with metrics $p,p'$, $U\subset S$, and $\alpha:{\cal F}|U\to {\cal F'}| U$ is a homomorphism. The norm $\|\alpha\|$ of $\alpha$ is a function $U\to [0,\infty]$
$$\|\alpha\|(s)=\inf\{a\in[0,\infty): p'(\alpha\varphi)(s)\le a p(\varphi)(s)\  \text{for all}\  \varphi\in{\cal F}_s\}.
$$
By a homomorphism $({\cal F}|U,p)\to ({\cal F'}|U,p')$ we mean a homomorphism $\alpha:{\cal F} | U\to {\cal F'}|U$ such that $\|\alpha\|:U\to [0,\infty)$ is continuous. Germs of such homomorphisms constitute a sheaf of $\cal O$-modules denoted {\bf{Hom}}$(({\cal F},p),({\cal F'},p'))$. The norm $\|\ \|$ induces a metric on this sheaf in the sense of Definition 2.1.

Suppose $E, E'\to S$ are holomorphic Banach bundles endowed with metrics $q,q'$, and $({\cal F}, p), ({\cal F'},p')$ are the sheaves of their holomorphic sections. Any bundle homomorphism $A:E\to E'$ induces, by composition, a homomorphism $({\cal F}, p)\to ({\cal F'}, p')$. 
Conversely (but here we need dim$S<\infty$):

\begin{lemma}Any homomorphism $\alpha:{\cal F}\to {\cal F'}$ is induced by a bundle homomorphism $E\to E'$. In particular, for such sheaves any germ of a homorphism $\cal F\to\cal F'$ is automatically in ${\bf{Hom}}(({\cal F},p),({\cal F'},p'))$.
\end{lemma}
This is [M, Theorem 1.1] when $E, E'$ are trivial, whence the general case follows immediately.

For $s\in S$ we denote by ${\cal J}(s)\subset\cal O$ the subsheaf of germs vanishing at $s$. If ${\cal F}\to S$ is a sheaf of $\cal O$-modules and $p$ a metric on it, (2.4) implies that $p(\varphi)(s)=0$ when $\varphi\in({\cal J}(s){\cal F})_s$. Hence $p$ induces a seminorm $p^s$ on the vector space ${\cal F}^s={\cal F}_s/({\cal J}(s){\cal F})_s$. Any homomorphism $\alpha:({\cal F},p)\to({\cal F'},p')$ induces continuous linear maps $\alpha^s:({\cal F}^s,p^s)\to({\cal F'}^{s},p'^{s})$. Suppose $({\cal F}, p)$ is the sheaf of holomorphic sections of a holomorphic Banach bundle $E\to S$ endowed with a metric $q$. It is easy to check that then $({\cal F}^s,p^s)$ is canonically isometrically isomorphic to $(E_s,q|E_s)$.

To define the curvature of a metric on a sheaf $\cal F$ we need to measure pointwise the Levi form of a not necessarily smooth function. Let us start with an open $D\subset\bC$, an upper semicontinuous $u:D\to\bR$, and define
\begin{equation}
\Lambda u(z_0)=\limsup_{r\to 0} r^{-2}\big(\int^1_0 u(z_0+re^{2\pi i\tau})d\tau -u(z_0)\big),
\end{equation}
$z_0\in D$. The definition is implicit in \cite{S}. Thus $\Lambda u=\partial^2u/\partial z\partial \bar z$ whenever $u\in C^2(D)$, and then $\limsup$ turns into $\lim$. If now $S$ is a complex manifold, $u:S\to \bR$ is upper semicontinuous, and $\xi\in T^{1,0}_s S$, define
$$
\xi\bar \xi u=\inf\{\Lambda(u\circ f)(0)\in [-\infty,\infty]\}
$$
the $\inf$ taken over holomorphic maps $f$ of some neighborhood of $0\in \bC$ to $S$ that map $\partial/\partial z$ to $\xi$.

\begin{lemma}An upper semicontinuous $u:S\to\bR$ is plurisubharmonic if and only if $\xi\bar\xi u\ge 0$ for all $\xi\in T^{1,0} S$.
\end{lemma}

\begin{proof}The ``only if'' direction follows from the definition of plurisubharmonicity. It suffices to prove the converse when $S\subset \bC$ is an open subset. So $\Lambda u\ge 0$ in $S$, which according to Saks implies $u$ is subharmonic. (Saks obtained his result from a slightly weaker theorem of Littlewood, see \cite{L,S}. I am greatful to A. Weitsman, who pointed me to those papers.)
\end{proof}

Returning to a sheaf ${\cF}\to S$ of $\cO$-modules endowed with a metric $p$, let $\xi\in T^{1,0}_sS$ and $v\in {\cF}^s={\cF}_s/({\cJ}(s){\cF})_s$. If $p^s(v)\ne 0$, define $K_\xi(v)\in [-\infty,\infty]$ by
\begin{equation}
-K_\xi(v)=2 \inf \{\xi\bar\xi \log p(\varphi)(s):\varphi\in v\}.
\end{equation}
When $({\cF}, p)$ is the sheaf of sections of a holomorphic Hilbert bundle $E\to S$, endowed with a hermitian metric of class $C^2$, (2.3) in conjunction with the Cauchy--Schwarz inequality shows
$$K_\xi(v)=h(v,R(\xi,\bar\xi)v)/h(v,v).$$
(We identified ${\cF}^s\approx E_s$.) Since Griffiths positivity/negativity of a hermitian metric means $K_\xi(v)>0$, repectively $<0$, for $\xi, v\ne 0$, we call $K_\xi(v)$ for general $({\cF},p)$ the Griffiths curvature.

We can now formulate our main result, a generalization of Theorem 1.1:

\begin{theorem}Let ${\cF},{\cF'}\to S$ be sheaves of $\cO$-modules endowed with metrics $p,p'$ and let $K,K'$ denote their Griffiths curvature. Suppose $({\cF},p)$ is the sheaf of holomorphic sections of a holomorphic Hilbert bundle with a hermitian metric of class $C^2$. If a homomorphism $\alpha:({\cF},p)\to ({\cF'},p')$ decreases curvature in the sense that for $s\in S$
$$K'_\xi(\alpha^sv)\le K_\xi(v),\qquad\xi\in T^{1,0}_s S,\  v\in {\cF}^s,\quad p'(\alpha^s v)\ne 0,$$
then $\log\|\alpha\|$ is plurisubharmonic.
\end{theorem}

The theorem has the potential to say something about homomorphisms $({\cF},p)\to({\cF'}, p')$ even when $({\cF},p)$ is not the sheaf of sections of a hermitian Hilbert bundle, but can be locally approximated by such. Approximable sheaves arise naturally in certain direct image constructions, and in \cite{B1} Berndtsson already successfully applied approximation to study curvature of such direct images.

\begin{corollary}Let $({\cF},p)$ and $({\cF'},p')$ be as in Theorem 2.4. Suppose for $s\in S$ and $\xi\in T^{1,0}_s S$
$$\inf\{K_\xi(v):v\in {\cF}^s, p^s(v)\ne 0\}\ge\sup\{K'_\xi(v'):v'\in {\cF'}^s, p'^s(v')\ne 0\}.$$
Then the metric $||\cdot||$ on the sheaf ${\textbf{\Hom}}(({\cF},p),({\cF'},p'))$ has seminegative Griffiths curvature.
\end{corollary}

\section{The proof of Theorem 2.4}

We can assume that $S$ is an open subset of $\bC$, and $({\cF},p)$ is the sheaf of holomorphic sections of a trivial Hilbert bundle $E=S\times W\to S$. Here $(W, \langle\ ,\ \rangle)$ is a Hilbert space. The metric $h$ on $E$ is induced by a $C^2$ map $P:S\to {\text{End}} \, W$,
$$h((s,w_1),(s,w_2))=\langle P(s)w_1,w_2\rangle,\quad s\in S,\quad w_1,w_2\in W.$$
The connection form $\theta$ evaluated on $\xi=\partial/\partial s$ is $A=P^{-1}\partial P/\partial s$. Thus if $f:S\to W$ is holomorphic, and a section $\varphi$ of $E$ is given by $\varphi(s)=(s,f(s))$, then
\begin{equation}
\nabla_\xi\varphi=\big(\cdot,\frac{\partial f}{\partial s}+Af\big).
\end{equation}
Given $s_0\in S$ and a nonzero $w\in W$, take $f(s)=w+(s_0-s)A(s_0)w$, so that 
$$
\frac{\partial f(s)}{\partial s} +A(s)f(s)=\big(A(s)-A(s_0)+(s_0-s)A(s)A(s_0)\big)w.
$$
The corresponding section $\varphi$ satisfies
\begin{equation}
|p(\varphi)(s)-p(\varphi)(s_0)|, \; p(\nabla_\xi\varphi)(s)\le C|s-s_0|p(\varphi)(s_0),
\end{equation}
if $|s-s_0|<\var$, with some $\var, C$ that depend only on $s_0$ but not on $w$, i.e. not on $\varphi$. Assuming $p'(\alpha\varphi)(s)\ne 0$, by (2.6)
$$
2\Lambda(\log p'(\alpha\varphi))(s)\ge -K'_\xi(\alpha^s\varphi^s).
$$
On the other hand, $2\log p(\varphi)=\log h(\varphi,\varphi)$ is $C^2$ in a neighborhood of $s_0$, and by (2.3), (3.2)
$$
2\Lambda(\log p(\varphi))(s)=2\frac{\partial^2(\log p(\varphi))(s)}{\partial s\partial{\bar s}}\le -K_{\xi}(\varphi^s)+C_1|s-s_0|^2,
$$
if $|s-s_0|<\var_1=\min(\var, 1/2 C)$, and $C_1=2C^2$.
Therefore by the assumption of the theorem when $|s-s_0|<\var_1$
$$
\Lambda\big\{\log\frac{p'(\alpha\varphi)(s)}{p(\varphi)(s)}+C_1|s-s_0|^4\big\}\ge 0,
$$
i.e., the function in braces is subharmonic according to Lemma 2.3. 
Since $\|\alpha\|\geq p'(\alpha\varphi)/ p(\varphi)$, for $0<r<\var_1$
\begin{equation}
\begin{aligned}
\int^1_0\log\|\alpha\|(s_0+re^{2\pi i\tau})d\tau &\geq  \int^1_0\log \frac{p'(\alpha\varphi)}{p(\varphi)}(s_0+re^{2\pi i\tau})d\tau\\
 &\geq \log \frac{p'(\alpha\varphi)(s_0)}{p(\alpha\varphi)(s_0)}-C_1r^4.
\end{aligned} 
\end{equation}

Assuming $\|\alpha\|(s_0)>0$, $\cF$ has sections $\psi$ for which $\log\bigl(p'(\alpha\psi)(s_0)/p(\psi)(s_0)\bigr)$ is arbitrarily close to $\log\|\alpha\|(s_0)$. Since both $p(\psi)(s_0)$ and $p'(\alpha\psi)(s_0)$ depend only on $\psi^{s_0}$, these sections $\psi$  can be chosen from among the $\varphi$ introduced above. Hence (3.3) implies
$$
\int^1_0\log\|\alpha\|(s_0+re^{2\pi i\tau})d\tau\ge\log\|\alpha\|(s_0)-C_1r^4,
$$
and so $\Lambda\log\|\alpha\|(s_0)\ge 0$. This being true for all $s_0\in S$ as long as $\|\alpha\|(s_0)>0$, by Lemma 2.3 $\log\|\alpha\|$ is indeed subharmonic.

It is possible to generalize Theorem 2.4 to $({\cF},p)$ coming from a Banach, rather than Hilbert, bundle $\pi:E\to S$, endowed with a metric $q:E\to [0,\infty)$. By analyzing what is needed to adapt the proof given above one obtains a perhaps unattractive condition for $(E,q)$. If 
$(W,\|\ \|)$ is a Banach space, write $W^0=W\setminus\{0\}$, and for a Banach bundle $E$ write $E^0$ for $E$ minus its zero section. If the norm $\|\ \|$ is of class $C^1$ on $W^0$, it determines a map 
\begin{equation}
W^0\ni w\mapsto\partial\|\ \|\big|_w \in W^*.
\end{equation}
The dual $E^*$ of $E$ is also a holomorphic Banach bundle, and similarly, if the metric $q$ of $E$ is $C^k$ on $E^0$, applying (3.4) fiberwise we obtain a map $\gamma:E^0\to E^*$, of class $C^{k-1}$.

Our condition on $(E,q)$ will be first, that $q$ be $C^2$ on $E^0$; second, that for $s_0\in S$ there be a neighborhood $U\subset S$, a Banach space $W$ with unit sphere $\Sigma\subset W$, and a surjective $C^1$ map 
$$F:U\times W^0\to \pi^{-1}U\setminus {\text {zero}\,\,\,  \text{section}}$$
that is uniformly $C^1$ on $U\times\Sigma$; and third, that $F(\cdot, w)$ should be holomorphic for $w\in W^0$, and 
\[
\bar\partial_s\gamma\bigl(F(s,w)\bigr)=0\qquad \text{at}\quad s=s_0.
\]

Thus Theorem 2.4 remains true for $({\cF},p)$ the sheaf of sections of a holomorphic Banach bundle $(E,q)\to S$ that satisfies the above condition. The role of $\varphi$ in the proof given above would be played by $F(\cdot,w)$ in the proof of the generalization.

For example, let $(X, \mu)$ be a measure space, $a\in[2,\infty)$, and $W=L^a(X,\mu)$. Let furthermore $\rho:S\times X\to\bR$ be a function such that $\rho(s,\cdot)$ is measurable for all $s\in S$, $\rho(\cdot,x)$ is twice continuously differentiable for fixed $x\in X$, and the derivatives up to order $2$ (when computed in coordinate charts on $S$) are uniformly bounded and uniformly equicontinuous on compact subsets of $S$  as $x\in X$ varies. Then on the trivial bundle $E=S\times W\to S$
\[
q(s,w)=\left(\int_X|w(x)|^a e^{\rho(s,x)}d\mu(x)\right)^{1/a}
\]
defines a metric that satisfies the above conditions. Assuming $S\subset \bC^n$, a little calculation shows that one can choose
\[
F(s,w)=w+\frac{\partial \rho}{\partial s}\bigg|_{s=s_0}(s-s_0)\frac wa.
\]

\end{document}